\documentclass[reqno, 10pt]{amsart}
\usepackage{amssymb,amsmath,amsthm,amsfonts,color}

\usepackage{mathrsfs,dsfont}
\usepackage[left=3cm,right=3cm,top=3cm,bottom=3cm]{geometry}
\usepackage[colorlinks=true, pdfstartview=FitV, linkcolor=blue, 
            citecolor=blue, urlcolor=blue]{hyperref}
\numberwithin{equation}{section}
\theoremstyle{plain}
\newtheorem{theorem}{Theorem}[section]
\newtheorem{proposition}[theorem]{Proposition}
\newtheorem{lemma}[theorem]{Lemma}

\theoremstyle{definition}

\newcommand\R{\mathbb R}
\newcommand\M{\mathbb M}
\newcommand\N{\mathbb N}
\newcommand\Oo{\mathcal{O}}

\newcommand\iO{\int_\Omega}

\newcommand\iQ{\int_Q}

\newcommand\PP{\mathbb{P}}
\newcommand\QQ{\mathbb{Q}}
\newcommand\rd{\R^d}
\newcommand\Rm{\R^m}
\newcommand\rl{\R^l}

\newcommand\limn{\lim_{n\to +\infty}}
\newcommand\liminfn{\liminf_{n\to +\infty}}
\newcommand\ep{\varepsilon}
\newcommand{\scal}[2]{\langle #1,\,#2 \rangle}

\newcommand\wk{\rightharpoonup}

\newcommand\Rn{\mathbb{R}^N}

\newcommand\pdeor{\mathscr{A}}

\newcommand{\A}{\mathbb{A}}
\newcommand{\cx}[0]{\mathcal{C}_{x}}
\newcommand{\qa}[0]{Q_{\pdeor(x)}}
\newcommand{\iq}[0]{\int_Q}

\newcommand\be[1]{\begin{equation}\label{#1}}
\newcommand\ee{\end{equation}}

\title[Relaxation for $\pdeor$-quasiconvexity with variable coefficients] {Relaxation of $p$-growth integral functionals under space-dependent differential constraints}
\author[E. Davoli] {Elisa Davoli} 
\address[Elisa Davoli]{Faculty of Mathematics\\ University of Vienna\\Oskar-Morgenstern Platz 1\\A-1090 Vienna, Austria}
\email[E. Davoli]{elisa.davoli@univie.ac.at}
\author[I. Fonseca] {Irene Fonseca} 
\address[Irene Fonseca]{Department of Mathematics\\ Carnegie Mellon University\\Forbes Avenue\\Pittsburgh PA 15213, USA}
\email[I. Fonseca]{fonseca@andrew.cmu.edu}

\begin{document} 
\vskip .2truecm
\begin{abstract}
\small{A representation formula for the relaxation of integral energies
$$(u,v)\mapsto\iO f(x,u(x),v(x))\,dx,$$
is obtained, where $f$ satisfies $p$-growth assumptions, $1<p<+\infty$, and the fields $v$ are subjected to space-dependent first order linear differential constraints in the framework of $\pdeor$-quasiconvexity with variable coefficients. 
}
\end{abstract}
\maketitle

\section{Introduction}
The analysis of constrained relaxation problems is a central question in materials science. Many applications in continuum mechanics and, in particular, in magnetoelasticity, rely on the characterization of minimizers of non-convex multiple integrals of the type
$$
u\mapsto \int_{\Omega}f(x,u(x),\nabla u(x),\dots,\nabla^k u(x))\,dx
$$
or
\be{eq:ex2}
(u,v)\mapsto \int_{\Omega}f(x,u(x),v(x))\,dx,
\ee
where $\Omega$ is an open, bounded subset of $\R^N$, $u:\Omega\to \R^m,\,m\in \N$, and the fields $v:\Omega\to \R^d,\,d\in\N$, satisfy partial differential constraints of the type ``$\pdeor v=0$" other than ${\rm curl}\,v=0$ (see e.g. \cite{benesova.kruzik, fonseca.dacorogna}).\\

In this paper we provide a representation formula for the relaxation of non-convex integral energies of the form \eqref{eq:ex2}, in the case in which the energy density $f$ satisfies $p$-growth assumptions, and the fields $v$ are subjected to linear first-order space-dependent differential constraints.\\

The natural framework to study this family of relaxation problems is within the theory of $\pdeor$-quasiconvexity with variable coefficients. In order to present this notion, we need to introduce some notation. 

For $i=1\cdots,N$, let $A^i\in C^{\infty}(\R^N;\M^{l\times d})\cap W^{1,\infty}(\R^N;\M^{l\times d})$, let $1<p<+\infty$, and consider the differential operator 
$$\pdeor:L^p(\Omega;\R^d)\to W^{-1,p}(\Omega;\R^l),\quad d,l\in \N,$$
defined as
\be{eq:def-operator}\pdeor v:=\sum_{i=1}^N A^i(x)\frac{\partial v(x)}{\partial x_i}\ee
for every $v\in L^p(\Omega; \R^d)$, where \eqref{eq:def-operator} is to be interpreted in the sense of distributions. Assume that the symbol $\mathbb{A}:\R^N\times \R^N\to \M^{l\times d}$,

$$\mathbb{A}(x,w):=\sum_{i=1}^N A^i(x)w_i\quad\text{for }(x,w)\in \R^N\times \R^N,$$
satisfies the uniform constant rank condition (see \cite{murat})
  \begin{equation}
 \label{cr}
 \text{rank } \mathbb{A}(x,w)=r\quad\text{for every }x\in \R^N \text{and }w\in\mathbb{S}^{n-1}.
 \end{equation}
Let $Q$ be the unit cube in $\R^N$ with sides parallel to the coordinate axis, i.e.,
$$Q:=\Big(-\frac12,\frac12\Big).$$
Denote by $C^{\infty}_{\rm per}(\R^N;\R^m)$ the set of $\R^m$-valued smooth maps that are $Q$-periodic in $\R^N$,
and for every $x\in \Omega$ consider the set
$$\cx:=\Big\{w\in C^{\infty}_{\rm per}(\R^N;\R^m):\,\int_Q{w(y)\,dy}=0,\,\text{and } \sum_{i=1}^N A^i(x)\frac{\partial w(y)}{\partial y_i}=0 \Big\}.$$
Let $f:\Omega\times\R^m\times \R^d \to [0,+\infty)$ be a Carath\'eodory function. The $\pdeor-$quasiconvex envelope of $f(x,u,\cdot)$ for $x\in\Omega$ and $u\in\R^m$ is defined for $\xi\in \R^d$ as 
$$\qa f(x,u,\xi):=\inf\Big\{\iq f(x,u,\xi+w(y))\,dy:\,w\in\cx\Big\}.$$
We say that $f$ is $\pdeor-$quasiconvex if $f(x,u,\xi)=\qa f(x,u,\xi)$ for a.e. $x\in\Omega$, and for all $u\in\R^m$ and $\xi\in \R^d$.\\

The notion of $\pdeor$-quasiconvexity was first introduced by B. Dacorogna in \cite{dacorogna}, and extensively characterized in \cite{fonseca.muller} by I. Fonseca and S. M\"uller for operators $\pdeor$ defined as in \eqref{eq:def-operator}, satisfying the constant rank condition \eqref{cr}, and having constant coefficients, 
$$A^i(x)\equiv A^i\in \M^{l\times d}\quad\text{for every }x\in \R^N,\,i=1,\dots,N.$$
In that paper the authors proved (see \cite[Theorems 3.6 and 3.7 ]{fonseca.muller}) that under $p$-growth assumptions on the energy density $f$, $\pdeor$-quasiconvexity is necessary and sufficient for the lower-semicontinuity of integral functionals
$$I(u,v):=\int_{\Omega}f(x,u(x),v(x))\,dx\quad\text{for every }(u,v)\in L^p(\Omega;\R^m)\times L^p(\Omega;\R^d)$$
along sequences $(u^n,v^n)$ satisfying $u^n\to u$ in
measure, $v^n\wk v$ in $L^p(\Omega;\R^d)$, and $\pdeor v^n\to 0$ in $W^{-1,p}(\Omega)$. We remark that in the framework $\pdeor ={\rm curl}$, i.e., when $v^n=\nabla\phi^n$ for some	$\phi^n\in	W^{ 1,p} (\Omega; \R^m )$, $d=n\times m$, $\pdeor$-quasiconvexity reduces to Morrey's notion of quasiconvexity.\\

 The analysis of properties of $\pdeor-$quasiconvexity for operators with constant coefficients was extended in the subsequent paper \cite{braides.fonseca.leoni}, where A. Braides, I. Fonseca and G. Leoni provided an integral representation formula for relaxation problems under $p$-growth assumptions on the energy density, and 
 presented (via $\Gamma$-convergence) homogenization results for periodic integrands evaluated along $\pdeor-$free fields. These homogenization results were later generalized in \cite{fonseca.kromer}, where I. Fonseca and S. Kr\"omer worked under weaker assumptions on the energy density $f$. In \cite{kreisbeck.kromer, kreisbeck.rindler}, simultaneous homogenization and dimension reduction was studied in the framework of $\pdeor$-quasiconvexity with constant coefficients. Oscillations and concentrations generated by $\pdeor$-free mappings are the subject of \cite{fonseca.kruzik}. Very recently an analysis of the case in which the energy density is nonpositive has been carried out in \cite{kramer.kromer.kruzik.patho}, and applications to the theory of compressible Euler systems have been studied in \cite{chiodaroli.feireisl.kreml.wiedemann}. A parallel analysis for operators with constant coefficients and under linear growth assumptions for the energy density has been developed in \cite{arroyo, baia.chermisi.matias.santos, fonseca.muller.leoni,  matias.morandotti.santos}. A very general characterization in this setting has been obtained in \cite{arroyo.dephilippis.rindler}, following the new insight in \cite{dephilippis.rindler}.
 
 The theory of $\pdeor$-quasiconvexity for operators with variable coefficients has been characterized by P. Santos in \cite{santos}. Homogenization results in this setting have been obtained in \cite{davoli.fonseca} and \cite{davoli.fonseca2}.\\

This paper is devoted to proving a representation result for the relaxation of integral energies in the framework of $\pdeor$-quasiconvexity with variable coefficients. To be precise, let $1<p,q<+\infty$, $d,m,l\in \N$, and consider a Carath\'eodory function
 $f:\Omega\times\Rm\times \rd\to [0,+\infty)$ satisfying
$$\textrm{(H)}\quad 0\leq f(x,u,v)\leq C(1+|u|^p+|v|^q),\quad1<p,q<+\infty,$$
for a.e. $x\in\Omega$, and all $(u,v)\in\Rm\times\rd$, with $C>0$. 

Denoting by $\Oo(\Omega)$ the collection of open subsets of $\Omega$, for every $D\in\mathcal{O}(\Omega)$, $u\in L^p(\Omega;\R^m)$ and $v\in L^q(\Omega;\rd)$ with $\pdeor v=0$, we define
\begin{multline}
\label{eq:def-i-intro}
\mathcal{I}((u,v),D):=\inf\Big\{\liminf_{n\to +\infty}\int_{D}f(x,u_n(x),v_n(x)):\, u_n\to u\quad\text{strongly in } L^p(\Omega;\R^m),\\
v_n\wk v\quad\text{weakly in }L^q(\Omega;\R^d)\,\text{and }\pdeor v_n\to 0\quad\text{strongly in }W^{-1,q}(\Omega;\R^l)\Big\}.
\end{multline}
Our main result is the following.
\begin{theorem}
\label{thm:main}
Let $\pdeor$ be a first order differential operator with variable coefficients, satisfying \eqref{cr}.
Let $f:\Omega\times\Rm\times \rd\to [0,+\infty)$ be a Carath\'eodory function satisfying {\rm(H)}.
Then,
$$\int_D \qa f(x,u(x),v(x))\,dx=\mathcal{I}((u,v),D)$$
for all $D\in\Oo(\Omega), u\in L^p(\Omega;\Rm)$ and $v\in L^q(\Omega;\rd)$ with $\pdeor v=0$.
\end{theorem}

Adopting the ``blow-up" method introduced in \cite{fonseca.muller93}, the proof of the theorem consists in showing that the functional $\mathcal{I}((u,v),\cdot)$ is the trace of a Radon measure absolutely continuous with respect to the restriction of the Lebesgue measure $\mathcal{L}^N$ to $\Omega$, and proving that for a.e. $x\in\Omega$ the Radon-Nicodym derivative $\frac{d\mathcal{I}((u,v)\cdot)(x)}{d\mathcal{L}^N}$ coincides with the $\pdeor-$quasiconvex envelope of $f$.

 The arguments used are a combination of the ideas from \cite[Theorem 1.1]{braides.fonseca.leoni} and from \cite{santos}. The main difference with \cite[Theorem 1.1]{braides.fonseca.leoni}, which reduces to our setting in the case in which the operator $\pdeor$ has constant coefficients, is in the fact that while defining the operator $\mathcal{I}$ in \eqref{eq:def-i-intro} we can not work with exact solutions of the PDE, but instead we need to study sequences of asymptotically $\pdeor-$vanishing fields. As pointed out in \cite{santos}, in the case of variable coefficients the natural framework is the context of pseudo-differential operators. In this setting, we don't know how to project directly onto the kernel of the differential constraint, but we are able to construct an ``approximate" projection operator $P$ such that for every field $v\in L^p$, the $W^{-1,p}$ norm of $\pdeor Pv$ is controlled by the $W^{-1,p}$ norm of $v$ itself (we refer to \cite[Subsection 2.1]{santos} for a detailed explanation of this issue and to the references therein for a treatment of the main properties of pseudo-differential operators). For the same reason, in the proof of the inequality
 $$\frac{d\mathcal{I}((u,v)\cdot)(x)}{d\mathcal{L}^N}\leq \qa f(x,u(x),v(x))\quad\text{for a.e. }x\in\Omega,$$
 an equi-integrability argument is needed (see Proposition \ref{prop:same-limit}). We also point out that the representation formula in Theorem \ref{thm:main} was obtained in a simplified setting in \cite{davoli.fonseca2} as a corollary of the main homogenization result. Here we provide an alternative, direct proof, which does not rely on homogenization techniques.\\

The paper is organized as follows: in Section \ref{section:prel} we establish the main assumptions on the differential operator $\pdeor$ and we recall some preliminary results on $\pdeor-$quasiconvexity with variable coefficients. Section \ref{section:envelope} is devoted to the proof of Theorem \ref{thm:main}.\\

\noindent\textbf{Notation}\\
Throughout the paper $\Omega\subset \R^N$ is a bounded open set, $1<p,q<+\infty$, $\mathcal{O}(\Omega)$ is the set of open subsets of $\Omega$,  $Q$ denotes the unit cube in $\R^N$, $Q(x_0,r)$ and $B(x_0,r)$ are, respectively, the open cube and the open ball in $\Rn$, with center $x_0$ and radius $r$. Given an exponent $1<q<+\infty$, we denote by $q'$ its conjugate exponent, i.e., $q'\in (1,+\infty)$ is such that
$$\frac{1}{q}+\frac{1}{q'}=1.$$
Whenever a map $v\in L^q, C^{\infty},\cdots$ is $Q-$periodic, that is
$$v(x+e_i)=v(x)\quad i=1,\cdots, N,$$
for a.e. $x\in \R^N$, $\{e_1,\cdots,e_N\}$ being the standard basis of $\R^N$, we write $v\in L^q_{\rm per}, C^{\infty}_{\rm per},\dots$ We implicitly identify the spaces $L^q(Q)$ and $L^q_{\rm per}(\R^N)$.

We adopt the convention that $C$ will denote a generic constant, whose value may change from line to line in the same formula.

\section{Preliminary results}
\label{section:prel}
In this section we introduce the main assumptions on the differential operator $\pdeor$ and we recall some preliminary results about $\pdeor-$quasiconvexity.

For $i=1,\cdots, N$, $x\in \R^N$, consider the linear operators $A^i(x)\in \M^{l\times d}$, with $A^i\in C^{\infty}(\R^N;\M^{l\times d})\cap W^{1,\infty}(\R^N;\M^{l\times d})$. For every $v\in L^q(\Omega;\R^d)$ we set
$$\pdeor v:= \sum_{i=1}^N A^i(x)\frac{\partial v(x)}{\partial x_i}\in W^{-1,q}(\Omega;\R^l).$$

The symbol $\A:\R^N\times \R^N\setminus \{0\}\to \M^{l\times d}$ associated to the differential operator $\pdeor$ is
$$\A(x,\lambda):=\sum_{i=1}^N A^i(x)\lambda_i\in \M^{l\times d}$$
for every $x\in \R^N$, $\lambda\in \R^N\setminus \{0\}$.
We assume that $\pdeor$ satisfies the following \emph{uniform constant rank condition}:
\be{eq:constant-rank-condition}
\text{rank }\Big(\sum_{i=1}^N A^i(x)\lambda_i\Big)=r\quad\text{for all }x\in\R^N\text{ and }\lambda\in \Rn \setminus \{0\}.
\ee
For every $x\in\Rn$, $\lambda\in \Rn\setminus\{0\}$, let $\PP(x,\lambda):\rd\to\rd$ be the linear projection on Ker $\A(x,\lambda)$, and let $\QQ(x,\lambda):\rl\to\rd$ be the linear operator given by
\begin{eqnarray*}
&&\QQ(x,\lambda)\A(x,\lambda)v:=v-\PP(x,\lambda)v\quad\text{for all }v\in\rd,\\
&&\QQ(x,\lambda)\xi=0\quad\text{if }\xi\notin \text{Range }\A(x,\lambda).
\end{eqnarray*}
The main properties of $\PP(\cdot,\cdot)$ and $\QQ(\cdot,\cdot)$ are recalled in the following proposition (see \emph{e.g.} \cite[Subsection 2.1]{santos}).
\begin{proposition}
\label{prop:properties-P-Q}
Under the constant rank condition \eqref{eq:constant-rank-condition}, for every $x\in\Rn$ the operators $\PP(x,\cdot)$ and $\QQ(x,\cdot)$ are, respectively, $0$-homogeneous and $(-1)$-homogeneous. In addition, $\PP\in$ $C^{\infty}(\Rn\times \Rn\setminus\{0\};\M^{d\times d})$ and $\QQ\in C^{\infty}(\Rn\times\Rn\setminus\{0\};\M^{d\times l})$.
\end{proposition}

Let $\eta\in C^{\infty}_c(\Omega;[0,1])$, $\eta=1$ in ${\Omega'}$ for some ${\Omega'}\subset\subset \Omega$. We denote by $\A_{\eta}$ the symbol
\be{eq:def-A-eta}
\A_{\eta}(x,\lambda):=\sum_{i=1}^N \eta(x)A^i(x)\lambda_i,
\ee
for every $x\in\Rn$, $\lambda\in\Rn\setminus\{0\}$, and by $\pdeor_{\eta}$ the corresponding pseudo-differential operator (see \cite[Subsection 2.1]{santos} for an overview of the main properties of pseudo-differential operators). Let $\chi\in C^{\infty}(\R^+;\R)$ be such that $\chi(|\lambda|)=0$ for $|\lambda|<1$ and $\chi(|\lambda|)=1$ for $|\lambda|>2$. Let also $P_{\eta}$ be the operator associated to the symbol
\be{eq:def-P-eta}
\PP_{\eta}(x,\lambda):=\eta^2(x)\PP(x,\lambda)\chi(|\lambda|)
\ee
for every $x\in\Rn$, $\lambda\in\Rn\setminus\{0\}$. The following proposition (see \cite[Theorem 2.2 and Subsection 2.1]{santos}) collects the main properties of the operators $P_{\eta}$ and $\pdeor_{\eta}$.
\begin{proposition}
\label{prop:santos-properties}
Let $1<q<+\infty$, and let $\pdeor_{\eta}$ and $P_{\eta}$ be the pseudo-differential operators associated with the symbols \eqref{eq:def-A-eta} and \eqref{eq:def-P-eta}, respectively. Then there exists a constant $C$ such that
\begin{equation}
\label{eq:santos-prop01}
\|P_{\eta}v\|_{L^q(\Omega;\R^d)}\leq C\|v\|_{L^q(\Omega;\rd)}
\end{equation}
for every $v\in L^q(\Omega;\rd)$, and
\begin{align*}
&\|P_{\eta}v\|_{W^{-1,q}(\Omega;\R^d)}\leq C\|v\|_{W^{-1,q}(\Omega;\rd)},\\
&\|v-P_{\eta}v\|_{L^q(\Omega;\R^d)}\leq C\big(\|\pdeor_{\eta}v\|_{W^{-1,q}(\Omega;\rl)}+\|v\|_{W^{-1,q}(\Omega;\rd)}\big),\\
&\|\pdeor_{\eta}P_{\eta}v\|_{W^{-1,q}(\Omega;\rl)}\leq C\|v\|_{W^{-1,q}(\Omega;\rd)}
\end{align*}
for every $v\in W^{-1,q}(\Omega;\R^d)$.
\end{proposition}

\section{Proof of Theorem \ref{thm:main}}
\label{section:envelope}
Before proving Theorem \ref{thm:main} we state and prove a decomposition lemma, which generalizes  \cite[Lemma 2.15]{fonseca.muller} to the case of operators with variable coefficients.
\begin{lemma}
\label{lemma:equiint}
Let $1<q<+\infty$. Let $\pdeor$ be a first order differential operator with variable coefficients, satisfying \eqref{eq:constant-rank-condition}. Let $v\in L^q(\Omega;\rd)$, and let $\{v_n\}$ be a bounded sequence in $L^q(\Omega;\rd)$ such that
\begin{align*}
&v_n\wk v\quad\text{weakly in }L^q(\Omega;\rd),\\
&\pdeor v_n\to 0\quad\text{strongly in }W^{-1,q}(\Omega;\rl),\\
&\{v_n\}\text{ generates the Young measure }\nu.
\end{align*}
Then, there exists a $q$-equiintegrable sequence $\{\tilde{v}_n\}\subset L^q(\Omega;\rd)$ such that 
\begin{eqnarray}
&&\label{eq:1vn} \pdeor \tilde{v}_n\to 0\quad\text{strongly in }W^{-1,s}(\Omega;\rl)\quad\text{for every }1<s<q,\\
&&\nonumber\iO \tilde{v}_n(x)\,dx=\iO v(x)\,dx,\\
&&\label{eq:3vn} \tilde{v}_n-v_n\to 0\quad\text{strongly in }L^s(\Omega;\rd)\quad\text{for every }1<s<q,\\
&&\label{eq:4vn} \tilde{v}_n\wk v\quad\text{weakly in }L^q(\Omega;\rd).
\end{eqnarray}
In addition, if $\Omega\subset Q$ then we can construct the sequence $\{\tilde{v}^n\}$ so that $\tilde{v}_n-v\in L^q_{\text{per}}(\Rn;\rd)$ for every $n\in \N$.
\end{lemma}
\begin{proof}
Arguing as in the first part of \cite[Proof of Theorem 1.1]{santos}, we construct a $q$-equiintegrable sequence $\{\hat{v}_n\}$ satisfying \eqref{eq:1vn}, \eqref{eq:3vn} and \eqref{eq:4vn}.
The conclusion follows by setting $\tilde{v}_n:=\hat{v}_n-\iO\hat{v}_n(x)\,dx+\iO v(x)\,dx$.

In the case in which $\Omega\subset Q$, let $\{\varphi^i\}$ be a sequence of cut-off functions in $Q$ with $0\leq \varphi^i\leq 1$ in $Q$, such that $\varphi^i=0$ on $Q\setminus \Omega$ and $\varphi^i\to 1$ pointwise in $\Omega$. Define $w^i_n:=\varphi^i(\hat{v}_n-v)$. By \eqref{eq:4vn} for every $\psi\in L^{q'}(\Omega;\rd)$ we have
$$\lim_{i\to +\infty}\lim_{n\to +\infty}\iO w^i_n(x)\psi(x)\,dx=0.$$
By \eqref{eq:1vn}, \eqref{eq:3vn}, and the compact embedding of $L^q(\Omega;\R^d)$ into $W^{-1,q}(\Omega;\R^d)$, there holds
$$\pdeor w^i_n=\varphi^i \pdeor \hat{v}_n+\Bigg(\sum_{j=1}^N A^j\frac{\partial \varphi^i}{\partial x_j}\Bigg)\hat{v}_n\to 0\quad\text{strongly in }W^{-1,s}(\Omega;\rl)$$
as $n\to +\infty$, for every $1<s<q$. Extending the maps $w^i_n$ outside $Q$ by periodicity, by the metrizability of the weak topology on bounded sets and by Attouch's diagonalization lemma (see \cite[Lemma 1.15 and Corollary 1.16]{attouch}), we obtain a sequence 
$$w_n:=w^{i(n)}_n,$$
 with $\{w_n\}\subset L^q_{\text{per}}(\Rn;\rd)$, and such that $w_n+v$ satisfies
\eqref{eq:1vn}, \eqref{eq:3vn} and \eqref{eq:4vn}.
The thesis follows by setting 
$$\tilde{v}_n:=w_n-\iO w_n(x)\,dx+v.$$
\end{proof}
The following proposition will allow us to neglect vanishing perturbations of $q$-equiintegrable sequences.
\begin{proposition}
\label{prop:same-limit}
For every $n\in \N$, let $f_n:Q\times \R^d\to [0,+\infty)$ be a continuous function. Assume that there exists a constant $C>0$ such that, for $q>1$,
\be{eq:unif-cont-fn}
\sup_{n\in \N} f_n(y,\xi)\leq C(1+|\xi|^q)\quad\text{for every }y\in Q\text{ and }\xi\in \R^d,
\ee
and that the sequence $\{f_n(y,\cdot)\}$ is equicontinuous in $\R^d$, uniformly in $y$.
Let $\{w_n\}$ be a $q$-equiintegrable sequence in $L^q(Q;\R^d)$, and let $\{v_n\}\subset L^q(Q;\R^d)$ be such that
\be{eq:vn-to-zero}
v_n\to 0\quad\text{strongly in }L^q(Q;\R^d).
\ee
Then
$$\lim_{n\to +\infty}\Big|\int_Q f_n\big(y,w_n(y)\big)\,dy-\int_Q f_n\big(y,v_n(y)+w_n(y)\big)\,dy\Big|=0.$$
\end{proposition}
\begin{proof}
Fix $\eta>0$. In view of \eqref{eq:vn-to-zero}, the sequence $\{C(1+|v_n|^q+|w_n|^q)\}$ is equiintegrable in $Q$, thus there exists $0<\ep<\frac{\eta}{3}$ such that
\be{eq:prop1-new}
\sup_{n\in \N} \int_{A} C\big(1+|v_n(y)|^q+|w_n(y)|^q\big)\,dy<\frac{\eta}{3}
\ee
for every $A\subset Q$ with $|A|<\ep$.
By the $q$-equiintegrability of $\{w_n\}$ and $\{v_n\}$, and by Chebyshev's inequality there holds
$$\big|Q\cap \big(\{|w_n|>M\}\cup \{|v_n|>M\}\big)\big|\leq \frac{1}{M^q}\int_{Q}(|w_n(y)|^q+|v_n(y)|^q)\,dy\leq \frac{C}{M^q}$$
for every $n\in \N$. Therefore, there exists $M_0$ satisfying
\be{eq:prop2-new}
\sup_{n\in \N}\big|Q\cap \big(\{|w_n|>M_0\}\cup \{|v_n|>M_0\}\big)\big|\leq \frac{\ep}{2}.
\ee
By the uniform equicontinuity of the sequence $\{f_n(y,\cdot)\}$, there exists $\delta>0$ such that, for every $\xi_1,\xi_2\in \overline{B(0,M_0)}$, with $|\xi_1-\xi_2|<\delta$, we have
\be{eq:prop3-new}
\sup_{y\in Q}|f_n(y,\xi_1)-f_n(y,\xi_2)|<\ep
\ee
for every $n\in \N$.
By \eqref{eq:vn-to-zero} and Egoroff's theorem, there exists a set $E_{\ep}\subset Q$, $|E_{\ep}|<\frac{\ep}{2}$, such that
$$v_n\to 0\quad\text{uniformly in }Q\setminus E_{\ep},$$
and, in particular, 
\be{eq:prop5-new}
|v_n(x)|<\delta\quad\text{for a.e. }x\in Q\setminus E_{\ep},
\ee
for every $n\geq n_0$, for some $n_0\in \N$. 

We observe that
\begin{align}
\label{eq:prop2bis-new}\int_Q f_n(y, v_n(y)+w_n(y))\,dy&=\int_{Q\cap \{|w_n|\leq M_0\}\cap  \{|v_n|\leq M_0\}}f_n(y, v_n(y)+w_n(y))\,dy\\
&\nonumber\quad+\int_{Q\cap (\{|w_n|> M_0\}\cup  \{|v_n|> M_0\})}f_n(y, v_n(y)+w_n(y))\,dy.
\end{align}
The first term in the right-hand side of \eqref{eq:prop2bis-new} can be further decomposed as
\begin{align*}
&\int_{Q\cap \{|w_n|\leq M_0\}\cap  \{|v_n|\leq M_0\}}f_n(y, v_n(y)+w_n(y))\,dy\\
&\quad=\int_{(Q\setminus E_{\ep})\cap \{|w_n|\leq M_0\}\cap  \{|v_n|\leq M_0\}}f_n(y, v_n(y)+w_n(y))\,dy\\
&\qquad+\int_{E_{\ep}\cap \{|w_n|\leq M_0\}\cap  \{|v_n|\leq M_0\}}f_n(y, v_n(y)+w_n(y))\,dy\\
&\quad=\int_{(Q\setminus E_{\ep})\cap \{|w_n|\leq M_0\}\cap  \{|v_n|\leq M_0\}}f_n(y,w_n(y))\,dy\\
&\qquad+\int_{(Q\setminus E_{\ep})\cap \{|w_n|\leq M_0\}\cap  \{|v_n|\leq M_0\}}\big(f_n(y, v_n(y)+w_n(y))-f_n(y,w_n(y))\big)\,dy\\
&\qquad+\int_{E_{\ep}\cap \{|w_n|\leq M_0\}\cap  \{|v_n|\leq M_0\}}f_n(y, v_n(y)+w_n(y))\,dy\\
&\quad=\int_Q f_n(y,w_n(y))\,dy-\int_{E_{\ep}\cap \{|w_n|\leq M_0\}\cap  \{|v_n|\leq M_0\}}f_n(y,w_n(y))\,dy\\
&\qquad-\int_{Q\cap(\{|w_n|> M_0\}\cup  \{|v_n|> M_0\})}f_n(y,w_n(y))\,dy\\
&\qquad+\int_{(Q\setminus E_{\ep})\cap \{|w_n|\leq M_0\}\cap  \{|v_n|\leq M_0\}}\big(f_n(y, v_n(y)+w_n(y))-f_n(y,w_n(y))\big)\,dy\\
&\qquad+\int_{E_{\ep}\cap \{|w_n|\leq M_0\}\cap  \{|v_n|\leq M_0\}}f_n(y, v_n(y)+w_n(y))\,dy.
\end{align*}
We observe that by \eqref{eq:prop2-new} 
$$|E_{\ep}\cup (\{|w_n|> M_0\}\cup  \{|v_n|> M_0\})|<\ep.$$
Hence, for $n\geq n_0$, by \eqref{eq:unif-cont-fn}, \eqref{eq:prop1-new}, \eqref{eq:prop3-new}, and \eqref{eq:prop5-new} we deduce the estimate
\begin{align}
\label{eq:prop6-new}
&\Big|\int_Q f_n(y,w_n(y))\,dy-\int_Q f_n(y,v_n(y)+w_n(y))\,dy\Big|\\
&\quad\nonumber\leq \ep+\int_{{E_{\ep}\cup (\{|w_n|> M_0\}\cup  \{|v_n|> M_0\})}}2C(1+|w_n(y)|^p+|v_n(y)|^p)\,dy\leq \ep +\frac{2\eta}{3}.
\end{align}
The thesis follows by the arbitrariness of $\eta$.
\end{proof}
We now prove our main result.
\begin{proof}[Proof of Theorem \ref{thm:main}]
The proof is subdivided into 4 steps. Steps 1 and 2 follow along the lines of \cite[Proof of Theorem 1.1]{braides.fonseca.leoni}.  Step 3 is obtained by modifying \cite[Lemma 3.5]{braides.fonseca.leoni}, whereas Step 4 follows by adapting an argument in \cite[Proof of Theorem 1.2]{santos}. We only outline the main ideas of Steps 1 and 2 for convenience of the reader, whilst we provide more details for Steps 3 and 4.\\
\emph{Step 1}:\\
The first step consists in showing that
\begin{align*}
\mathcal{I}((u,v),D)=\inf&\Big\{\liminfn \int_D f(x,u(x),v_n(x))\,dx:\,\{v_n\}\text{ is }q-\text{equiintegrable },\\&\qquad\pdeor v_n\to 0\text{ strongly in }W^{-1,s}(D;\rl)\text{ for every }1<s<q\\
&\qquad\text{ and }v_n\wk v\text{ weakly in }L^q(D;\rd)\Big\}.
\end{align*}
This identification is proved by adapting \cite[Proof of Lemma 3.1]{braides.fonseca.leoni}. The only difference is the application of Lemma \ref{lemma:equiint} instead of \cite[Proposition 2.3 (i)]{braides.fonseca.leoni}.\\
\emph{Step 2}:\\
The second step is the proof that $\mathcal{I}((u,v),\cdot)$ is the trace of a Radon measure absolutely continuous with respect to $\mathcal{L}^N\lfloor\Omega$. This follows as a straightforward adaptation of \cite[Lemma 3.4]{braides.fonseca.leoni}. The only modifications are due to the fact that \cite[Proposition 2.3 (i)]{braides.fonseca.leoni} and \cite[Lemma 3.1]{braides.fonseca.leoni} are now replaced by Lemma \ref{lemma:equiint} and Step 1.\\
\emph{Step 3}:\\
We claim that
\be{eq:ineq1}\frac{d\mathcal{I}((u,v),\cdot)}{d\mathcal{L}^N}(x_0)\geq Q_{\pdeor (x_0)}f(x_0,u(x_0),v(x_0))\quad\text{for a.e. }x_0\in\Omega.\ee

Indeed, since $g(x,\xi):=f(x,u(x),\xi)$ is a Carath\'eodory function, by Scorza-Dragoni Theorem there exists a sequence of compact sets $K_j\subset \Omega$ such that 
$$|\Omega\setminus K_j|\leq \tfrac 1j$$
 and the restriction of $g$ to $K_j\times \rd$ is continuous. Hence, the set
\be{eq:defsetomega}
\omega:=\bigcup_{j=1}^{+\infty}(K_j\cap K_j^{*})\cap \mathcal{L}(u,v),
\ee
where $K_j^{*}$ is the set of Lebesgue point for the characteristic function of $K_j$ and $\mathcal{L}(u,v)$ is the set of Lebesgue points of $u$ and $v$, is such that
$$|\Omega\setminus \omega|\leq |\Omega\setminus K_j|\leq \frac{1}{j}\quad\text{for every }j,$$
and so $|\Omega\setminus \omega|=0$.
Let $x_0\in\omega$ be such that
\be{eq:omega-1}
\lim_{r\to 0^{+}}\frac{1}{r^N}\int_{Q(x_0,r)}|u(x)-u(x_0)|^p\,dx=\lim_{r\to 0^{+}}\frac{1}{r^N}\int_{Q(x_0,r)}|v(x)-v(x_0)|^q\,dx=0,\ee
and
\be{eq:omega-2}
\frac{d\mathcal{I}((u,v),\cdot)}{d\mathcal{L}^N}(x_0)=\lim_{r\to 0^{+}}\frac{\mathcal{I}((u,v),Q(x_0,r))}{r^N}<+\infty,
\ee
where the sequence of radii $r$ is such that $\mathcal{I}((u,v),\partial Q(x_0,r))=0$ for every $r$. (Such a choice of the sequence is possible due to Step 2). 

By Step 1, for every $r$ there exists a $q-$equiintegrable sequence $\{v_{n,r}\}$ such that 
\begin{align}
\nonumber&v_{n,r}\wk v\quad\text{weakly in }L^q(Q(x_0,r);\rd),\\
\label{eq:pdevnr}&\pdeor v_{n,r}\to 0\quad\text{strongly in }W^{-1,s}(Q(x_0,r);\rl)\text{ for every }1<s<q
\end{align}
as $n\to +\infty$, and
$$\limn\int_{Q(x_0,r)} g(x,v_{n,r}(x))\,dx\leq \mathcal{I}((u,v), Q(x_0,r))+r^{N+1}.$$
A change of variables yields
$$\frac{d\mathcal{I}((u,v),\cdot)}{d\mathcal{L}^N}(x_0)\geq \liminf_{r\to 0^{+}}\lim_{n\to +\infty}\iQ g(x_0+ry,v(x_0)+w_{n,r}(y))\,dy,$$
where $$w_{n,r}(y):=v_{n,r}(x_0+ry)-v(x_0)\quad\text{for a.e. }y\in Q.$$
Arguing as in \cite[Proof of Lemma 3.5]{braides.fonseca.leoni}, H\"older's inequality and a change of variables imply 
\be{eq:wnr1}w_{n,r}\wk 0\quad\text{weakly in }L^q(Q;\rd)\ee
as $n\to +\infty$ and $r\to 0^+$, in this order. We claim that
\be{eq:wnr2}\pdeor(x_0+r\cdot)w_{n,r}\to 0\quad\text{strongly in }W^{-1,s}(Q;\rl),\ee
as $n\to +\infty$, for every $r$ and every $1<s<q$.

Indeed, let $\varphi\in W^{1,s'}_0(Q;\rd)$. There holds
\begin{align*}
&\scal{\pdeor(x_0+r\cdot)w_{n,r}}{\varphi}_{W^{-1,s}(Q;\R^l),W^{1,s'}_0(Q;\R^l)}=
-\sum_{i=1}^N \Bigg\{r\iQ\frac{\partial A^i(x_0+ry)}{\partial x_i}v_{n,r}(x_0+ry)\cdot\varphi(y)\,dy\\
&\qquad+\iQ A^i(x_0+ry)v_{n,r}(x_0+ry)\cdot\ \frac{\partial \varphi(y)}{\partial y_i}\,dy\Bigg\}\\
&\quad=-\sum_{i=1}^N\Bigg\{\frac{1}{r^{N-1}}\int_{Q(x_0,r)}\frac{\partial A^i(x)}{\partial x_i}v_{n,r}(x)\cdot \psi_r(x)\,dx+\frac{1}{r^{N-1}}\int_{Q(x_0,r)}A^i(x)v_{n,r}(x)\cdot \frac{\partial \psi_r(x)}{\partial x_i}\,dx\Bigg\}\\
&\quad=\frac{1}{r^{N-1}}\scal{\pdeor v_{n,r}}{\psi_r}_{W^{-1,s}(Q(x_0,r);\R^l),W^{1,s'}_0(Q(x_0,r);\R^l)},
\end{align*}
where $\psi_r(x):=\varphi\big(\frac{x-x_0}{r}\big)$ for a.e. $x\in Q(x_0,r)$. Since $\psi_r\in W^{1,s'}_0(Q(x_0,r);\rd)$ and 
$$\|\psi_r\|_{W^{1,s'}_0(Q(x_0,r);\rd)}\leq C(r)\|\varphi\|_{W^{1,s'}_0(Q;\rd)},$$
we obtain the estimate
$$\|\pdeor(x_0+r\cdot)w_{n,r}\|_{W^{-1,s}(Q;\rl)}\leq C(r)\|\pdeor v_{n,r}\|_{W^{-1,s}(Q(x_0,r);\rl)}.$$
Claim \eqref{eq:wnr2} follows by \eqref{eq:pdevnr}.

In view of \eqref{eq:wnr1} and \eqref{eq:wnr2}, a diagonalization procedure yields a $q-$equiintegrable sequence $\{\hat{w}_k\}\subset L^q(Q;\rd)$ satisfying
\begin{align}
\label{eq:whatk}& \hat{w}_k\wk 0\quad\text{weakly in }L^q(Q;\rd),\\
\label{eq:whatkpde} &\pdeor(x_0+r_k\cdot)\hat{w}_k\to 0\quad\text{strongly in }W^{-1,s}(Q;\rl)\quad\text{for every }1<s<q,
 \end{align}
 and
 \begin{equation}
 \label{eq:whatkbelow}
 \frac{d\mathcal{I}((u,v),\cdot)}{d\mathcal{L}^N}(x_0)\geq \liminf_{k\to +\infty}\iQ g(x_0+r_k y,v(x_0)+\hat{w}_{k}(y))\,dy.
\end{equation}
 For every $\varphi\in W^{1,s'}_0(Q;\R^l)$, $1<s<q$, there holds
 \begin{align*}
&\scal{(\pdeor(x_0+r_k\cdot)-\pdeor(x_0))\hat{w}_k}{\varphi}_{W^{-1,s}(Q;\rl),W^{1,s'}_0(Q;\R^l)}\\
&\quad=-\sum_{i=1}^N\Bigg[r_k \int_Q \frac{\partial A^i(x_0+r_k y)}{\partial x_i}\hat{w}_k(y)\cdot \varphi(y)\,dy+\int_Q (A^i(x_0+r_ky)-A^i(x_0))\hat{w}_k(y)\cdot\frac{\partial \varphi(y)}{\partial y_i}\,dy\Bigg].
 \end{align*}
 
 Thus,
$$\|(\pdeor(x_0+r_k\cdot)-\pdeor(x_0))\hat{w}_k\|_{W^{-1,s}(Q;\rl)}\leq r_k \sum_{i=1}^N \|A^i\|_{W^{1,\infty}(\R^N;\R^{l\times d})}\|\hat{w}_k\|_{L^q(Q;\R^d)}$$
for every $1<s<q$. By \eqref{eq:whatk} and \eqref{eq:whatkpde} we conclude that
 \begin{equation}
 \label{eq:pdefixedwhatk}\pdeor(x_0)\hat{w}_k\to 0\quad\text{strongly in }W^{-1,s}(Q;\rl)\quad\text{for every }1<s<q.
 \end{equation}
 In view of \eqref{eq:whatk} and \eqref{eq:pdefixedwhatk}, an adaptation of \cite[Corollary 3.3]{braides.fonseca.leoni} yields a $q-$equiintegrable sequence $\{w_k\}$ such that
 \begin{align}
 &\nonumber w_k\wk 0\quad\text{weakly in }L^q(Q;\rd),\\
 &\nonumber \iQ w_k(y)\,dy=0\quad\text{for every }k,\\
 &\label{eq:wkker}\pdeor(x_0)w_k=0\quad\text{for every }k,
 \end{align}
 and
 \be{eq:quasi-liminf}\liminf_{k\to +\infty}\iQ g(x_0,v(x_0)+w_k(y))\,dy\leq \liminf_{k\to +\infty}\iQ g(x_0+r_k y,v(x_0)+\hat{w}_k(y))\,dy.\ee
 Finally, by combining \eqref{eq:whatkbelow}, \eqref{eq:wkker}, and \eqref{eq:quasi-liminf}, and by the definition of $\pdeor$-quasiconvex envelope for operators with constant coefficients, we obtain
 \begin{align*}
\frac{d\mathcal{I}((u,v),\cdot)}{d\mathcal{L}^N}(x_0)&\geq \liminf_{k\to +\infty}\iQ g(x_0,v(x_0)+w_k(y))\,dy\\
&=\liminf_{k\to +\infty}\iQ f(x_0, u(x_0), v(x_0)+w_k(y))\,dy\geq Q_{\pdeor(x_0)} f(x_0, u(x_0), v(x_0))
 \end{align*}
 for a.e. $x_0\in\Omega$. This concludes the proof of Claim \eqref{eq:ineq1}.\\
 \emph{Step 4:}\\
 To complete the proof of the theorem we need to show that
 \be{eq:limsup-ineq-relax}
 \frac{d\mathcal{I}((u,v),\cdot)}{d\mathcal{L}^N}(x_0)\leq Q_{\pdeor(x_0)} f(x_0, u(x_0), v(x_0))\quad\text{for a.e. }x_0\in \Omega.
 \ee
 To this aim, let $\mu>0$, and $x_0\in\omega$ be such that \eqref{eq:omega-1} and \eqref{eq:omega-2} hold. Let $w\in C^{\infty}_{\rm per}(\Rn;\rd)$ be such that
 \be{eq:properties-w}\iQ w(y)\,dy=0,\quad\pdeor(x_0)w=0,\ee
 and
 \be{eq:below-inf}
 \iQ f(x_0,u(x_0), v(x_0)+w(y))\,dy\leq Q_{\pdeor(x_0)}f(x_0,u(x_0),v(x_0))+\mu.
 \ee
 Let $\eta\in C^{\infty}_c(\Omega;[0,1])$ be such that $\eta\equiv 1$ in a neighborhood of $x_0$ and let $r$ be small enough so that
 \be{eq:inclusions}
 Q(x_0,r)\subset \{x:\eta(x)=1\}\quad\text{and }Q(x_0,2r)\subset \subset \Omega.
 \ee
 Consider a map $\varphi\in C^{\infty}_c(Q(x_0,r);[0,1])$ satisfying 
 \be{eq:lebesgue-meas-not-one}
 \mathcal{L}^N(Q(x_0,r)\cap\{\varphi\neq 1\})<\mu r^N,
 \ee and define
 \be{eq:defurm}
 z^r_m(x):=\varphi(x)w\Big(\frac{m(x-x_0)}{r}\Big)\quad\text{for }x\in\Rn.
 \ee
 We observe that $z^r_m\in L^q(\Omega;\rd)$, and for $\psi\in L^{q'}(\Omega;\R^d)$ we have
 \begin{align*}
 \int_{\Omega} z^r_m(x)\cdot \psi(x)\,dx&=\int_{\Omega}\varphi(x) w\Big(\frac{m(x-x_0)}{r}\Big)\cdot \psi(x)\,dx\\
 &=r^N \int_Q \varphi(x_0+ry)w(my)\cdot \psi(x_0+ry)\,dy.
 \end{align*}
 By \eqref{eq:properties-w} and by the Riemann-Lebesgue lemma we have
 \be{eq:wkconvurm}
 z^r_m\wk 0\quad\text{weakly in }L^q(\Omega;\rd)
 \ee
 as $m\to +\infty$. We claim that
 \be{eq:estimate-r}
 \limsup_{m\to +\infty}\|\pdeor_{\eta}z^r_m\|_{W^{-1,q}(\Omega;\R^l)}\leq Cr^{\frac{N}{q}+1},
 \ee
 where $\pdeor_{\eta}$ is the pseudo-differential operator defined in \eqref{eq:def-A-eta}. Indeed, by \eqref{eq:inclusions} we obtain
 \begin{align}
 \label{eq:proof-estimate-r} \pdeor_{\eta}z^r_m&=\pdeor z^r_m-\pdeor(x_0)z^r_m+\pdeor(x_0)z^r_m\\
&\nonumber=\sum_{i=1}^N \frac{\partial((A^i(x)-A^i(x_0))z^r_m(x))}{\partial x_i}+\sum_{i=1}^NA^i(x_0)\frac{\partial z^r_m(x)}{\partial x_i}-\sum_{i=1}^N\frac{\partial A^i(x)}{\partial x_i}z^r_m(x).
  \end{align}
  By the regularity of the operators $A^i$ and by a change of variables, the first term in the right-hand side of \eqref{eq:proof-estimate-r} is estimated as 
  \begin{align}
 & \label{eq:estimate-first-term}
  \Big\|\sum_{i=1}^N \frac{\partial((A^i(x)-A^i(x_0))z^r_m(x))}{\partial x_i}\Big\|_{W^{-1,q}(\Omega;\rl)}\\
 &\nonumber \leq \sum_{i=1}^N\Big\|(A^i(x)-A^i(x_0))\varphi(x)w\Big(\frac{m(x-x_0)}{r}\Big)\Big\|_{L^q(Q(x_0,r);\rl)}\\
 &\nonumber \leq \sum_{i=1}^N\|A^i\|_{W^{1,\infty}(\R^N;\R^{l\times d})}\|\varphi\|_{L^{\infty}(Q(x_0,r))}\|w(m\cdot)\|_{L^q(Q;\R^d)}r^{\frac{N}{q}+1}\leq Cr^{\frac{N}{q}+1}.
  \end{align}
  In view of \eqref{eq:properties-w} the second term in the right-hand side of \eqref{eq:proof-estimate-r} becomes
$$\sum_{i=1}^NA^i(x_0)\frac{\partial z^r_m(x)}{\partial x_i}=\sum_{i=1}^N A^i(x_0)\frac{\partial \varphi(x)}{\partial x_i}w\Big(\frac{m(x-x_0)}{r}\Big),$$
 and thus converges to zero weakly in $L^q(\Omega;\rl)$, as $m\to +\infty$, due to \eqref{eq:properties-w} and by the Riemann-Lebesgue lemma. Hence, 
  \be{eq:estimate-second-term}
  \Big\|\sum_{i=1}^NA^i(x_0)\frac{\partial z^r_m(x)}{\partial x_i}\Big\|_{W^{-1,q}(\Omega;\rl)}\to 0\quad\text{as }m\to +\infty
  \ee
  by the compact embedding of $L^q(\Omega;\rl)$ into $W^{-1,q}(\Omega;\rl)$. Finally, the third term in the right-hand side of \eqref{eq:proof-estimate-r} satisfies
  $$\sum_{i=1}^N\frac{\partial A^i(x)}{\partial x_i}z^r_m(x)=\sum_{i=1}^N\frac{\partial A^i(x)}{\partial x_i}\varphi(x)w\Big(\frac{m(x-x_0)}{r}\Big),$$
  which again converges to zero weakly in $L^q(\Omega;\rl)$, as $m\to +\infty$, owing again to \eqref{eq:properties-w} and the Riemann-Lebesgue lemma. Therefore,
  \be{eq:estimate-third-term}
  \Big\|\sum_{i=1}^N\frac{\partial A^i(x)}{\partial x_i}z^r_m(x)\Big\|_{W^{-1,q}(\Omega;\rl)}\to 0\quad\text{as }m\to +\infty.
  \ee
  Claim \eqref{eq:estimate-r} follows by combining \eqref{eq:estimate-first-term}--\eqref{eq:estimate-third-term}.
  
   Consider the maps 
   $$v^r_m:=P_{\eta}z^r_m,$$
   where $P_{\eta}$ is the projection operator introduced in \eqref{eq:def-P-eta}.
   By Proposition \ref{prop:santos-properties} we have
  \begin{align}
  &\label{eq:rev-star3}\|v^r_m\|_{L^q(Q(x_0,r);\rd)}\leq C\|z^r_m\|_{L^q(\Omega;\rd)},\\
  &\label{eq:rev-star}\|v^r_m\|_{W^{-1,q}(Q(x_0,r);\rd)}\leq C\|z^r_m\|_{W^{-1,q}(\Omega;\rd)},\\
  &\label{eq:rev-star2}\|\pdeor_{\eta}v^r_m\|_{W^{-1,q}(Q(x_0,r);\rl)}\leq C\|z^r_m\|_{W^{-1,q}(\Omega;\rd)},\\
  &\label{eq:rev-star4}\|v^r_m-z^r_m\|_{L^q(Q(x_0,r);\rd)}\leq C(\|\pdeor_{\eta}z^r_m\|_{W^{-1,q}(\Omega;\rl)}+\|z^r_m\|_{W^{-1,q}(\Omega;\R^d)}).
  \end{align}
  By \eqref{eq:wkconvurm} and \eqref{eq:rev-star3}, the sequence $\{v^r_m\}$ is uniformly bounded in $L^q(Q(x_0,r);\rd)$. Thus, there exists a map $v^r\in L^q(Q(x_0,r);\rd)$ such that, up to the extraction of a (not relabelled) subsequence, 
  \be{eq:rev-starstar}
  v^r_m\wk v^r\quad\text{weakly in }L^q(Q(x_0,r);\rd)
  \ee
  as $m\to +\infty$. Again by \eqref{eq:wkconvurm}, and by the compact embedding of $L^q$ into $W^{-1,q}$, we deduce that
  \be{eq:rev-point}
  z^r_m\to 0\quad\text{strongy in }W^{-1,q}(\Omega;\rd)
  \ee
  as $m\to +\infty$. Therefore, by combining \eqref{eq:rev-star} and \eqref{eq:rev-starstar}, we conclude that
  $$v^r_m\wk 0\quad\text{weakly in }L^q(Q(x_0,r);\rd)$$
  as $m\to +\infty$, and the convergence holds for the entire sequence. Additionally, by  \eqref{eq:inclusions}, \eqref{eq:rev-star2}, and \eqref{eq:rev-point}, we obtain
 $$\pdeor v^r_m=\pdeor_{\eta}v^r_m\to 0\quad\text{strongly in }W^{-1,q}(Q(x_0,r);\rl)$$
as $m\to +\infty$. Finally, by \eqref{eq:estimate-r}, \eqref{eq:rev-star4}, and \eqref{eq:rev-point}, there holds
\be{eq:rev-pointpoint}
\lim_{r\to 0}\lim_{m\to +\infty}r^{-\tfrac Nq}\|v^r_m-z^r_m\|_{L^q(Q(x_0,r);\rd)}=0.
\ee

  We recall that, since $x_0$ satisfies \eqref{eq:omega-2}, Step 1 yields
 \be{eq:NUM}\frac{d\mathcal{I}(u,v)}{d\mathcal{L}^N}(x_0)=\lim_{r\to 0^+}\frac{\mathcal{I}((u,v);Q(x_0,r))}{r^N}
  \leq \liminf_{r\to 0^+}\liminf_{m\to +\infty} \frac{1}{r^N}\int_{Q(x_0,r)}f(x,u(x),v(x)+v^r_m(x))\,dx.\ee
  We claim that
\be{eq:NUM-claim}\frac{d\mathcal{I}(u,v)}{d\mathcal{L}^N}(x_0)=\lim_{r\to 0^+}\frac{\mathcal{I}((u,v);Q(x_0,r))}{r^N}
  \leq \liminf_{r\to 0^+}\liminf_{m\to +\infty} \frac{1}{r^N}\int_{Q(x_0,r)}g(x,v(x)+z^r_m(x))\,dx,\ee
 where $g$ is the function introduced in Step 3. Indeed, for every $r\in \R$, consider the function $g^r:Q\times \R^d\to [0,+\infty)$ defined as
 $$g^r(y,\xi):=g(x_0+ry,\xi)\quad\text{for every }y\in Q,\xi\in \R^d.$$
 Since $x_0\in \omega$, by \eqref{eq:defsetomega} there exists $K_j$ such that $x_0\in K_j$. In particular, this yields the existence of $r_0>0$ such that for $r\leq r_0$, the maps $g^r$ are continuous on $Q\times \R^d$, and the family $\{g^r(y,\cdot)\}$ is equicontinuous in $\R^d$, uniformly with respect to $y$.
 A change of variables yields
 \begin{align*}
 &\frac{1}{r^N}\Big|\int_{Q(x_0,r)}f(x,u(x),v(x)+v^r_m(x))\,dx-\int_{Q(x_0,r)}f(x,u(x),v(x)+z^r_m(x))\,dx\Big|\\
 &=\Big|\int_{Q}g^r(y,v(x_0+ry)+v^r_m(x_0+ry))\,dy-\int_{Q}g^r(y,v(x_0+ry)+z^r_m(x_0+ry))\,dy\Big|.
 \end{align*}
 On the other hand, by \eqref{eq:rev-pointpoint} we have
 $$\lim_{r\to 0}\lim_{m\to +\infty}\|z^r_m(x_0+r\cdot)-v^r_m(x_0+r\cdot)\|_{L^q(Q;\R^d)}=\lim_{r\to 0}\lim_{m\to +\infty}r^{-\frac{N}{q}}\|z^r_m-v^r_m\|_{L^q(Q(x_0,r);\R^d)}= 0.$$
Therefore, by a diagonal procedure we extract a subsequence $\{m_r\}$ such that
\begin{align}
&\label{eq:newnew}
\limsup_{r\to 0}\limsup_{m\to +\infty}\Big|\int_{Q}g^r(y,v(x_0+ry)+v^r_m(x_0+ry))\,dy-\int_{Q}g^r(y,v(x_0+ry)+z^r_m(x_0+ry))\,dy\Big|\\
&\nonumber=\lim_{r\to 0}\Big|\int_{Q}g^r(y,v(x_0+ry)+v^r_{m_r}(x_0+ry))\,dy-\int_{Q}g^r(y,v(x_0+ry)+z^r_{m_r}(x_0+ry))\,dy\Big|,
\end{align}
 and
 $$z^r_{m_r}(x_0+r\cdot)-v^r_{m_r}(x_0+r\cdot)\to 0\quad\text{strongly in }L^q(Q;\R^d).$$
 In view of \eqref{eq:omega-1}, \eqref{eq:defurm} and the Riemann-Lebesgue lemma, the sequence $\{v(x_0+r\cdot)+z^r_{m_r}(x_0+r\cdot)\}$ is $q$-equiintegrable in $Q$. Hence, by (H) we are under the assumptions of Proposition \ref{prop:same-limit}, and we conclude that
 \be{eq:newnewnew}
 \lim_{r\to 0}\Big|\int_{Q}g^r(y,v(x_0+ry)+v^r_{m_r}(x_0+ry))\,dy-\int_{Q}g^r(y,v(x_0+ry)+z^r_{m_r}(x_0+ry))\,dy\Big|=0.
 \ee
 Claim \eqref{eq:NUM-claim} follows by combining \eqref{eq:newnew} with \eqref{eq:newnewnew}.

  Arguing as in \cite[Proof of Lemma 3.5]{braides.fonseca.leoni}, for every $x_0\in\omega$ (where $\omega$ is the set defined in \eqref{eq:defsetomega}) we have
  \begin{align*}
  &\liminf_{r\to 0^+}\liminf_{m\to +\infty} \frac{1}{r^N}\int_{Q(x_0,r)}f(x,u(x),v(x)+z^r_m(x))\,dx\\
  &\quad\leq \liminf_{r\to 0^+}\liminf_{m\to +\infty} \frac{1}{r^N}\int_{Q(x_0,r)}f(x_0,u(x_0),v(x_0)+z^r_m(x))\,dx,
  \end{align*}
  hence by \eqref{eq:NUM-claim} we deduce that
  $$\frac{d\mathcal{I}(u,v)}{d\mathcal{L}^N}(x_0)\leq \liminf_{r\to 0^+}\liminf_{m\to +\infty} \frac{1}{r^N}\int_{Q(x_0,r)}f(x_0,u(x_0),v(x_0)+z^r_m(x))\,dx.$$
  By \eqref{eq:defurm} we obtain
    \begin{align*}
  \frac{d\mathcal{I}(u,v)}{d\mathcal{L}^N}(x_0)&\leq \liminf_{r\to 0^+}\liminf_{m\to +\infty}\frac{1}{r^N}\int_{Q(x_0,r)}f(x_0,u(x_0),v(x_0)+z^r_m(x))\,dx\\
  &\leq \liminf_{r\to 0^+}\liminf_{m\to +\infty}\frac{1}{r^N}\Bigg\{\int_{Q(x_0,r)}f\Big(x_0,u(x_0),v(x_0)+w\Big(\frac{m(x-x_0)}{r}\Big)\Big)\,dx\\
  &\qquad+\int_{Q(x_0,r)\cap\{\varphi\neq 1\}}f\Big(x_0,u(x_0),v(x_0)+\varphi(x)w\Big(\frac{m(x-x_0)}{r}\Big)\Big)\,dx\Bigg\}.
  \end{align*}
 The growth assumption (H) and estimate \eqref{eq:lebesgue-meas-not-one} yield
  \begin{align}
  &\label{eq:a-f}\int_{Q(x_0,r)\cap\{\varphi\neq 1\}}f\Big(x_0,u(x_0),v(x_0)+\varphi(x)w\Big(\frac{m(x-x_0)}{r}\Big)\Big)\,dx\\
  &\nonumber\quad\leq C\int_{Q(x_0,r)\cap\{\varphi\neq 1\}}\Big(1+\Big|w\Big(\frac{m(x-x_0)}{r}\Big)\Big|^q\Big)\,dx\\
  &\nonumber\quad\leq C(1+\|w\|^q_{L^{\infty}(\Rn;\rd)})\mathcal{L}^N(Q(x_0,r)\cap\{\varphi\neq 1\})\leq C\mu r^N.
  \end{align}
  Thus, by \eqref{eq:a-f}, the periodicity of $w$, and Riemann-Lebesgue lemma, we deduce 
  \begin{align*}
   \frac{d\mathcal{I}(u,v)}{d\mathcal{L}^N}(x_0)&\leq C\mu+\liminf_{r\to 0^+}\liminf_{m\to +\infty}\frac{1}{r^N}\int_{Q(x_0,r)}f\Big(x_0,u(x_0),v(x_0)+w\Big(\frac{m(x-x_0)}{r}\Big)\Big)\,dx\\
&=C\mu+\liminf_{m\to +\infty}\int_{Q}f(x_0,u(x_0),v(x_0)+w(my))\,dy\\
&=C\mu+\int_{Q}f(x_0,u(x_0),v(x_0)+w(y))\,dy\\
&\leq C\mu+ Q_{\pdeor(x_0)}f(x_0,u(x_0),v(x_0)),
  \end{align*}
  where the last inequality is due to \eqref{eq:below-inf}. Letting $\mu\to 0^+$ we conclude \eqref{eq:limsup-ineq-relax}.
\end{proof}  
\section*{Acknowledgements}
The authors thank the Center for Nonlinear Analysis (NSF Grant No. DMS-0635983), where this research was carried out, and also acknowledge support of the National Science Foundation under the PIRE Grant No. OISE-0967140. The research of I. Fonseca and E. Davoli was funded by the National Science Foundation under Grant No. DMS- 0905778. E. Davoli acknowledges the support of the Austrian Science Fund (FWF) projects P 27052 and I 2375. The research of I. Fonseca was further partially supported by the National Science Foundation under Grant No. DMS-1411646.
\bibliographystyle{plain}
\bibliography{ed}
\end{document}